\documentclass{amsart}
\usepackage{
	amsmath, 
	amsthm, 
	amssymb,
	fancyhdr, 
	setspace,
	tikz
}
\usepackage[
	paper=letterpaper,
	left=1.5in,
	top=1in,
	bottom=1in,
	right=1.5in,
	includehead,
	includefoot
]
{geometry}

\usepackage{array}
\usepackage[all]{xy}
\usepackage{setspace}
\usepackage{mathrsfs}
\usepackage{bbm}
\usepackage{url}
\usepackage{rotating}
\usepackage{graphicx}
\usepackage{bold-extra}
\usepackage{fontenc}
\usepackage{lmodern}
\usepackage{ wasysym }

\usepackage[OT2,T1]{fontenc}
\DeclareSymbolFont{cyrletters}{OT2}{wncyr}{m}{n}
\DeclareMathSymbol{\Sha}{\mathalpha}{cyrletters}{"58}

\newcommand{\define}[1]{\textbf{\textit{#1}}}

\newcommand{\ZZ}{\mathbf Z}

\newcommand{\QQ}{\mathbf Q}

\newcommand{\NN}{\mathbf N}

\renewcommand{\hat}{\widehat}

\renewcommand{\Im}{\operatorname{Im}}

\newcommand{\GL}{\operatorname{GL}}

\newcommand{\SL}{\operatorname{SL}}
\newcommand{\PSL}{\operatorname{PSL}}

\newcommand{\Aut}{\operatorname{Aut}}

\newcommand{\Gal}{\operatorname{Gal}}

\newcommand{\OO}{\mathcal{O}}

\newtheorem{thm}{Theorem}[section]
\newtheorem{lem}[thm]{Lemma}

\newtheorem{prop}[thm]{Proposition}
\newtheorem{cor}[thm]{Corollary}

\theoremstyle{remark}
\newtheorem{rmk}[thm]{\bf Remark}
\numberwithin{equation}{section}

\newtheorem{dfn}[thm]{\bf Definition}

\begin{document}
\title{Elliptic Curves with maximally disjoint division fields}

\begin{abstract}
One of the many interesting algebraic objects associated to a given rational elliptic curve, $E$, is its full-torsion representation $\rho_E:\Gal(\bar{\QQ}/\QQ)\to\GL_2(\hat{\ZZ})$. Generalizing this idea, one can create another full-torsion Galois representation, $\rho_{(E_1,E_2)}:\Gal(\bar{\QQ}/\QQ)\to\left(\GL_2(\hat{\ZZ})\right)^2$ associated to a pair $(E_1,E_2)$ of rational elliptic curves. The goal of this paper is to provide an infinite number of concrete examples of pairs of elliptic curves whose associated full-torsion Galois representation $\rho_{(E_1,E_2)}$ has maximal image. The size of the image is inversely related to the size of the intersection of various division fields defined by $E_1$ and $E_2$. The representation $\rho_{(E_1,E_2)}$ has maximal image when these division fields are maximally disjoint, and most of the paper is devoted to studying these intersections.
\end{abstract}

\author{Harris B. Daniels}
\address[Harris B. Daniels]{Department of Mathematics
Amherst College Box 2239
P.O. 5000
Amherst, MA 01002-5000}
\email{hdaniels@amherst.edu}

\author{Jeffrey Hatley}
\address[Jeffrey Hatley]{Department of Mathematics
Bailey Hall 202
Union College
Schenectady, NY 12308}
\email{hatleyj@union.edu}

\author{James Ricci}
\address[James Ricci]{Department of Mathematics and Computer Science
Daemen College 
Duns Scotus 339 
4380 Main Street 
Amherst, NY 14226}
\email{jricci@daemen.edu}

\date{\today}


\maketitle

\section{Introduction}
Let $E$ be an elliptic curve defined over $\QQ$, let $\bar{\QQ}$ be a fixed algebraic closure of $\QQ$, and for each positive integer $n$ let
\[
E[n] = \left\{P\in E(\bar{\QQ}) : [n]P =\OO\right\}
\]
denote the $n$-torsion of $E$. It is a classical result that $E[n]$ is non-canonically isomorphic to $\ZZ/n\ZZ\times\ZZ/n\ZZ$ and the group $G_\QQ = \Gal(\bar{\QQ}/\QQ)$ acts on $E[n]$ component-wise. Therefore, we can construct a Galois representation associated to the $n$-torsion of $E$,
\[
\bar\rho_{E,n}: G_\QQ\to\Aut(E[n])\simeq\GL_2(\ZZ/n\ZZ).
\]
By choosing compatible bases and taking an inverse limit ordered by divisibility, we can construct the full-torsion representation associated to $E$,
\[
\rho_E:G_\QQ\to \GL_2(\hat{\ZZ}) \simeq \prod_p \GL_2(\ZZ_p),
\]
where the product is taken over all prime numbers.

A natural question is, how large can the image of $\rho_E$ be inside of $\GL_2(\hat{\ZZ})$? More specifically, can $\rho_E$ be surjective? With these questions in mind, we give the following definition:

\begin{dfn}
An integer $n\geq 2$ is said to be \emph{exceptional} for $E$ if $\bar{\rho}_{E,n}$ is not surjective. 
\end{dfn}

We can translate questions about the size of $\Im\rho_E$ into a question about which numbers are exceptional for $E$ and, for an exceptional $n$, how drastically $\rho_{E,n}$ fails to be surjective. It is a standard result that when $E$ is an elliptic curve with complex multiplication (CM), \emph{every} integer except for possibly 2 is exceptional for $E$. See \cite[Theorem 2.3]{ATAEC} for more detail. On the other hand, if $E$ is a rational elliptic curve that does not have CM, Serre showed in \cite{Serre72} that the index $[\GL_2(\hat{\ZZ}):\Im\rho_E]$ is finite. One implication of this is that for each elliptic curve there are only finitely many exceptional primes. Additionally, Serre proved the following theorem.

\begin{prop}\label{prop-index2}\cite[Proposition 22]{Serre72}
For any elliptic curve $E$ defined over $\QQ$, the image of $\rho_E:G_\QQ \to \GL_2(\hat{\ZZ})$ is contained in a group of index $2$ inside $\GL_2(\hat{\ZZ})$.
\end{prop}

This theorem implies that $\rho_E$ can never be surjective, and thus there exists at least one exceptional number $n$ (not necessarily prime). In the same paper, Serre gave two examples of elliptic curves whose image has index exactly 2 inside $\GL_2(\hat\ZZ)$, showing that this lower bound on the index of $\Im \rho_E$ is sharp. 

Following Lang and Trotter we give the following definition:

\begin{dfn}\label{dfn-serre_curve}
An elliptic curve $E/\QQ$ is called a \define{Serre curve} if $[\GL_2(\hat{\ZZ}):\Im\rho_E]=2$. 
\end{dfn}

Furthermore, there is no reason to restrict our attention to Galois representations associated to only \emph{one} elliptic curve. Given a pair of rational elliptic curves $(E_1,E_2)$ and a positive integer $n$, we can consider the action of $G_\QQ$ on $E_1[n]\times E_2[n]$ to get a new Galois representation 
\[
\bar\rho_{(E_1,E_2),n}: G_\QQ\to\ \left(\GL_2(\ZZ/n\ZZ)\right)^2,
\]
given by $\bar\rho_{(E_1,E_2),n}(\sigma) = (\bar\rho_{E_1,n}(\sigma),\bar\rho_{E_2,n}(\sigma))$. Just as before we can construct the full-torsion representation associated to the pair $(E_1,E_2)$
\[
\rho_{(E_1,E_2)}:G_\QQ\to \left( \GL_2(\hat{\ZZ}) \right)^2,
\]
and it is again natural to ask, how big can the image of $\rho_{(E_1,E_2)}$ be? 

There is a natural limitation on the size of the image of $\rho_{(E_1,E_2)}$ in $\GL_2(\hat\ZZ)$ coming from the Weil pairing. Given an elliptic curve $E/\QQ$, let $\QQ(E[n])$ be the field of definition of the $n$-torsion points of $E$. One consequence of the Weil pairing is that if $\zeta_n$ is a primitive $n$-th root of unity, then $\QQ(\zeta_n)\subset\QQ(E[n])$. Therefore, it must be that $\QQ(\zeta_n) \subset \QQ(E_1[n])\cap\QQ(E_2[n])$.

The action of an element in the Galois group on an $n$-th root of unity can be related to its image under $\bar\rho_{E,n}$ through the determinant. That is, given an elliptic curve $E/\QQ$, $\sigma\in G_\QQ$, and an $n$-th root of unity $\zeta_n$, it must \emph{always} be that 
\begin{equation}\label{eq-cyclotomic}
\sigma(\zeta_n) = \zeta_n^{\det(\bar\rho_{E,n}(\sigma))}.	
\end{equation}

Therefore, for each positive integer $n$, we define 
\[
D_n := \left\{(A,B)\in\Big(\GL_2(\ZZ/n\ZZ)\Big)^2: \det A =\det B\right\}
\]
and 
\[
D := \left\{(A,B)\in \left( \GL_2(\hat{\ZZ}) \right)^2: \det A = \det B  \right\}.
\]

With these definitions and the observations above we can see that for any pair of rational elliptic curves $(E_1,E_2)$ and positive integer $n$, the image of $\bar\rho_{(E_1,E_2),n}$ and $\rho_{(E_1,E_2)}$ must be contained inside of $D_n$ and $D$ respectively. Therefore, any result associated with the size of $\Im\rho_{(E_1,E_2)}$ should be formulated in terms of $[D:\Im\rho_{(E_1,E_2)}]$.

For any two elliptic curves $E_1$ and $E_2$ defined over $\QQ$, we have
\[
\Im\rho_{(E_1,E_2)} \subset (\Im\rho_{E_1}\times \Im\rho_{E_2})\cap D.
\]
Since the right-hand side has index at least 4 inside of $D$ (by Proposition \ref{prop-index2}), we give the following definition in the spirit of Definition \ref{dfn-serre_curve}:
\begin{dfn}
A pair $(E_1,E_2)$ is called a \define{Serre pair} if $[D: \Im\rho_{(E_1,E_2)}] = 4.$	
\end{dfn}

In \cite{Jones2013}, Jones shows that, in some appropriate sense, almost all pairs of elliptic curves are Serre pairs. The proof uses a multi-dimensional large sieve but provides no concrete examples of Serre pairs. In fact, there are no examples of Serre pairs in the current literature. The main goal of this paper is to rectify this deficiency by providing infinitely many examples of Serre pairs. The first step toward this goal is to find an infinite family of Serre curves since clearly any Serre pair must be a pair of Serre curves.

\begin{lem}\label{lem-ex_serre_curves}\cite[Example 8.2]{Daniels-SC}
Let $\ell$ be an odd prime with $\ell\neq 7$. Then the elliptic curve $$E_\ell: y^2 + xy = x^3 + \ell$$
is a Serre curve. 
\end{lem}

Using this lemma we will be able to construct the first examples of Serre pairs coming from the main theorem of this paper: 

\begin{thm}\label{thm-main}
Let $\ell_1$ and $\ell_2$ be odd primes not equal to $7$ such that $\gcd(432\ell_1^2+\ell_1, 432\ell_2^2+\ell_2) = 1$, and for $i=1,2$ let 
\[ 
E_{\ell_i}: y^2+xy = x^3+\ell_i.
\]
Then the pair $(E_{\ell_1},E_{\ell_2})$ is a Serre pair.
\end{thm}

In fact, we obtain the following slightly stronger result.

\begin{cor}\label{corollary-main}
Let $\ell_1$ be an odd prime different from $7$. Then there exist infinitely many primes $\ell_2$ such that the pair $(E_{\ell_1},E_{\ell_2})$ is a Serre pair.
\end{cor}

\begin{proof}

Let $\Delta=432 \ell_1^2 + \ell_1$ and suppose it factors as $\Delta = p_1^{e_1} \cdots p_n^{e_n}$. 
By Theorem \ref{thm-main}, it suffices to show that there exist infinitely many primes $\ell_2 \nmid \Delta$ such that 
\[
432 \ell_2 + 1 \not\equiv 0 \mod p_i \ \mathrm{ \text{for every} }\  i=1,\ldots,n.
\]

First notice that if $\ell_1 =3$, then by Dirichlet's theorem on primes in arithmetic progressions, there are infinitely many primes $\ell_2$ different from $3$ and $1297$ such that $\ell_2\not\equiv 3\bmod 1297.$ 


Otherwise, if $\ell_1 \neq 3$, then $432 \ell_1 \equiv -1$ is a unit modulo $\Delta$ and since each $p_i \mid \Delta$, we have
\[
432 \ell_2 + 1 \equiv 0 \mod p_i \ \Longrightarrow  \ \ell_2 \equiv \ell_1 \mod p_i.
\]
Therefore, it suffices to show that there are infinitely many $\ell_2$ such that $\ell_2 \not\equiv \ell_1  \mod p_i$  for all $1 \leq i \leq n$. By the Chinese remainder theorem, we can choose $x$ such that $x \not\equiv 0, \ell_1 \mod p_i$ for each $i$.  An application of Dirichlet's theorem on the sequence $\{x + (p_1 \cdots p_n)k\}_{k \in \NN}$  then guarantees the existence of infinitely many primes $\ell_2$ with the desired property.

\end{proof}

\begin{rmk}
The quantity $432\ell_i^2+\ell_i$ is the discriminant of the elliptic curve $E_i$. As we discuss below in Proposition \ref{prop-Neron-Ogg-Sha} and Lemma \ref{lem-ramification}, the hypothesis that $\gcd(432\ell_1^2+\ell_1, 432\ell_2^2+\ell_2) = 1$ imposes constraints on the ramification in the division fields associated to our elliptic curves. 

\end{rmk}

In order to prove this theorem we will need the following lemma:
\begin{lem}\label{lem-2_conditions}
Let $(E_1,E_2)$ be a pair of rational elliptic curves. If 
\begin{enumerate}
\item for each prime $p\geq 5$, $\Im\bar\rho_{(E_1,E_2),p} = D_p$, and
\item $\Im\bar\rho_{(E_1,E_2),36} = D_{36}$,
\end{enumerate}
then $(E_1,E_2)$ is a Serre Pair. 
\end{lem}
\begin{proof}
This follows immediately from \cite[Lemma 3.1]{Jones2013}.
\end{proof}

Lemma \ref{lem-2_conditions} gives us two concrete conditions that we use to verify our pairs of elliptic curves are in fact Serre pairs. 

\subsection{Notation and Outline}

Throughout the rest of this paper, fix two odd primes $\ell_1$ and $\ell_2$, both different from $7$, such that $\gcd(432\ell_1^2+\ell_1,432\ell_2^2+\ell_2)=1$. For $i=1,2$ we will write $$E_i : y^2 + xy = x^3 + \ell_i.$$ Then by Lemma \ref{lem-ex_serre_curves}, $E_1$ and $E_2$ are both Serre curves. In particular, as explained in \cite{Daniels-SC}, we have that
\[
\bar\rho_{E_i,p^n} : G_\QQ \to \GL_2(\ZZ / p^n\ZZ)
\] 
is surjective for every prime $p$ and every integer $n \geq 1$.

Our strategy is to use Lemma \ref{lem-2_conditions} to prove that $(E_1,E_2)$ is a Serre pair. Thus, our paper divides naturally into two main sections: a study of $\bar\rho_{(E_1,E_2),p}$ for all primes $p \geq 5$, and a separate study of $\bar\rho_{(E_1,E_2),36}$. In both cases, we interpret the conditions of Lemma \ref{lem-2_conditions} in terms of the Galois theory of the division fields associated to the Serre curves $E_i$. Let $K_i = \QQ(E_i[p^n])$ denote the Galois number field obtained by adjoining to $\QQ$ the coordinates of the $p^n$-torsion points of $E_i$. The Weil pairing forces the intersection $K_1 \cap K_2$ to be a non-trivial extension of $\QQ$; in particular, the intersection contains the $p^n$-cyclotomic field $\QQ(\zeta_{p^n})$. The main results of this paper state that, apart from the cyclotomic subextension, the division fields $K_1$ and $K_2$ are maximally disjoint for all primes $p$ and all integers $n \geq 1$. Theorem \ref{thm-main} then follows directly from the conditions found in Lemma \ref{lem-2_conditions}.     

\section{$p$-Division fields for $p \geq 5$}\label{section-p5andup}

For the entirety of this section fix a prime $p \geq 5$ and, since $\ell_1 \neq \ell_2$, assume without loss of generality that $p \neq \ell_1$.  Let $K_i=\QQ(E_i[p])$ denote the number field obtained by adjoining to $\QQ$ the $x$- and $y$-coordinates of the $p$-torsion points of $E_i$. 
Since $E_i$ is a Serre curve, we have
\[
\mathrm{Gal}(K_i/\QQ) \simeq \GL_2(\ZZ / p \ZZ). 
\]
As explained in the introduction, the Weil pairing forces the inclusion $\QQ(\zeta_p) \subset K_i$, where $\zeta_p$ denotes a primitive $p$-th root of unity and $\QQ(\zeta_p)$ denotes the $p$-cyclotomic extension of $\QQ$. Let $F=K_1 \cap K_2$ denote the intersection of the two division fields; then $F \supset \QQ(\zeta_p)$ is strictly larger than $\QQ$.

Recall that condition (1) of Lemma \ref{lem-2_conditions} states the following:
\begin{equation}\label{condition1}
\Im\bar\rho_{(E_1,E_2),p} = D_p, \mathrm{where}\ D_p = \left\{(A,B)\in\Big(\GL_2(\ZZ/p\ZZ)\Big)^2: \det A =\det B\right\}
\end{equation}
This condition can be interpreted using the Galois theoretic properties of the $K_i$, as we now describe.

First, recall that the determinant of $\bar\rho_{E_i,p}$ is the cyclotomic character $\chi_p$, which cuts out the cyclotomic extension $\QQ(\zeta_p)/\QQ$ via the canonical isomorphism  $\chi_p: \mathrm{Gal}(\QQ(\zeta_p) / \QQ) \xrightarrow{\sim} (\ZZ / p \ZZ)^\times$.

Now let $L=K_1 K_2$ denote the compositum of the division fields. Then $\mathrm{Gal}(L/\QQ)$ is a subgroup of the direct product $\GL_2(\ZZ / p\ZZ) \times \GL_2(\ZZ / p\ZZ)$. Since the intersection $F $ is a nontrivial extension of $\QQ$, $\mathrm{Gal}(L/\QQ)$ must be a \textit{proper} subgroup. The following result is well-known.

\begin{lem}[Goursat's Lemma]\label{lemma-Goursat}
Let $G_1$ and $G_2$ be groups, and let $H$ be a subgroup of the direct product $G_1 \times G_2$ such that the natural projections $\pi_1:H \to G_1$ and $\pi_2:H \to G_2$ are surjective. Let $N_1$ denote the kernel of $\pi_2$ and $N_2$ denote the kernel of $\pi_1$. Then regarding $N_i$ as a subgroup of $G_i$, the image of $H$ in $G_1/N_1 \times G_2/N_2$ is the graph of an isomorphism $G_1/N_1 \simeq G_2/N_2$. 
\end{lem}

\begin{proof}
See \cite[Lemma 5.2.1]{Ribet}.
\end{proof}

Write $G_i=\mathrm{Gal}(K_i/\QQ)$, and for the moment let $H=\mathrm{Gal}(L/\QQ)$. Goursat's lemma shows that $H$ is a certain fibered product of $G_1$ and $G_2$. Furthermore, since we have $G_1 \simeq G_2 \simeq \GL_2(\ZZ / p \ZZ)$, $H$ is determined by a unique normal subgroup $N$ of $\GL_2(\ZZ / p \ZZ)$.  For example, if $H$ were equal to the entire direct product $\GL_2(\ZZ / p \ZZ) \times \GL_2(\ZZ / p \ZZ)$, then we would have $N= \GL_2(\ZZ / p \ZZ)$, and the common fixed field $F= K_1 \cap K_2$ would be equal to $\QQ$.

Goursat's lemma thus gives the following Galois-theoretic interpretation of (\ref{condition1}): since $\det \bar\rho_{E_i,p} = \chi_p$ cuts out $\QQ(\zeta_p)$, we have 
\[
\Im\bar\rho_{(E_1,E_2),p} = D_p \Longleftrightarrow F = \QQ(\zeta_p).
\]
So (\ref{condition1}) is equivalent to the statement that $H$ is the fibered product of $G_1$ and $G_2$ over $\QQ(\zeta_p)$, which is equivalent to $K_1$ and $K_2$ being maximally disjoint. Our goal is now to show that $F = \QQ(\zeta_p).$

 To that end, let us now set $H:=\mathrm{Gal}(L/\QQ(\zeta_p))$. Figure \ref{diagram-p-atleast-5} illustrates the associated field diagram with edges labeled by Galois groups.

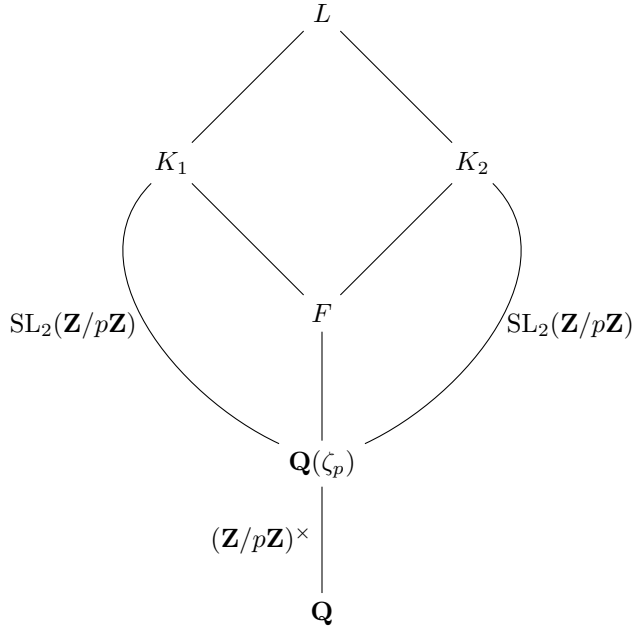
\begin{figure}
 \begin{tikzpicture}[node distance = 2cm, auto]
 	\centering
      \node (Q)       {$\QQ$};
      \node (Qzeta) [above of=Q]  {$\QQ(\zeta_p)$};
      \node (F)  [above of=Qzeta] {$F$};
      \node (K1) [above of=F, left of=F] {$K_1$};
      \node (K2) [above of=F, right of=F] {$K_2$};
      \node (L)  [above of=F, node distance = 4cm] {$L$};
      \draw[-] (Q) to node {$(\ZZ / p \ZZ)^\times$} (Qzeta);
      \draw[-] (Qzeta) to node {} (F);
      \draw[-] (F) to node {} (K1);
      \draw[-] (F) to node {} (K2);
      \draw (Qzeta) to [out=155,in=225]node[left,midway] {$\SL_2(\ZZ / p \ZZ)$} (K1);
      \draw (Qzeta) to [out=25,in=315]node[right,midway]  {$\SL_2(\ZZ / p \ZZ)$} (K2);
      \draw[-] (K1) to node {} (L);
      \draw[-] (K2) to node {} (L);
      \end{tikzpicture}
      \caption{Division fields for $p \geq 5$} \label{diagram-p-atleast-5}
\end{figure}

$H$ is a subgroup of the direct product $\SL_2(\ZZ / p \ZZ) \times \SL_2(\ZZ / p \ZZ)$, and we wish to show that $H \simeq (\SL_2(\ZZ / p \ZZ))^2$. Since $E_1$ and $E_2$ are Serre curves, the natural projections $H \to \SL_2(\ZZ /p \ZZ)$ are surjective, and Goursat's lemma implies that $H$ is determined by a normal subgroup $N \triangleleft \SL_2(\ZZ / p \ZZ)$. As in our previous discussion, we will have $F = \QQ(\zeta_p)$ precisely if $N = \SL_2(\ZZ / p \ZZ)$.

Before proving the main result of this section, we collect some lemmas on the ramification behavior of primes in the $K_i$. One computes that $E_{\ell_i}: y^2+xy = x^3+\ell_i$
has discriminant
\[
\Delta(E_{\ell_i})= - \ell_i (432 + \ell_i).
\]
Recall that the only primes of bad reduction for $E_i$ are those dividing $\Delta(E_i)$. The following result states that these are also the only primes other than $p$ which may ramify in $K_i/\QQ$.
\begin{prop}[Neron, Ogg, Shafarevich]\label{prop-Neron-Ogg-Sha}
Let $E$ be an elliptic curve over $\QQ$, and let $p$ be a rational prime. Then the following assertions are equivalent:
\begin{itemize}
\item $E$ has good reduction modulo $p$.
\item $p$ is unramified in $\QQ(E[n])/\QQ$ for all integers $n \geq 1$ with $\gcd(n,p)=1$.
\end{itemize}
\end{prop}
\begin{proof}
See \cite[VII. Theorem 7.1]{AEC}
\end{proof}

By hypothesis we have $\gcd(\Delta(E_1),\Delta(E_2))=1$, so $\ell_2$ does not ramify in $K_1$. The next lemma gives a lower bound on the ramification of $\ell_1$ in $K_1$.

\begin{lem}\label{lem-ramification}
Let $e_{\ell_i}$ denote the ramification index of $\ell_i$ in $K_i / \QQ$. Then $e_{\ell_i} \geq p$.
\end{lem}

\begin{proof}
This is worked out in detail in \cite[Section 3.2]{Lozano-Lundell} using the theory of Tate curves. For the proof, we drop the $i$ subscripts and write simply $E=E_i$ and $\ell=\ell_i$. First, note that the discriminant of $E$ is
\[
\Delta(E)= - \ell (432 + \ell).
\] 
In particular, the $\ell$-adic valuation of $\Delta(E)$ is 
\[
\nu_\ell(\Delta(E))= \begin{cases}
1 & \mathrm{if}\ \ell \neq 3  \\ 
2 & \mathrm{if}\ \ell = 3 
\end{cases}
\]
and $E$ has bad (split multiplicative) reduction at $\ell$. Our elliptic curve has $j$-invariant $j_E=\frac{1}{\Delta(E)}$, so in the notation of \cite{Lozano-Lundell} we have $\alpha_\ell=\nu_p(-\nu_\ell(j_E))=0$. By displayed equations (3.4)--(3.7) of \cite[Section 3.2]{Lozano-Lundell}, we have

\[ e_{\ell} = \begin{cases} 
      (p-1)p & \mathrm{if}\ p=\ell, \\
      p & \mathrm{if}\ p \neq \ell 
   \end{cases}.
\]
Thus, in either case we have $e_\ell \geq p$.

\end{proof}

We are now prepared to prove the following.

\begin{prop}\label{prop-p-atleast-5}
Let $N$ denote the kernel of (either) projection map $H \to \SL_2(\ZZ / p \ZZ)$. Then $N = \SL_2(\ZZ / p \ZZ)$, and consequently $\Im \bar \rho_{(E_1,E_2),p} = D_p$.
\end{prop}

\begin{proof}
Recall that we have a decomposition
\[
\SL_2(\ZZ / p \ZZ) \simeq \langle \pm I \rangle \times \PSL_2(\ZZ / p \ZZ),
\]
where $I$ denotes the identity matrix, and where the projective special linear group $\PSL_2(\ZZ / p \ZZ)$ is a simple group since $p \geq 5$ \cite[Proposition 5.1.7]{Adelmann}. Thus, $H$ is determined by a normal subgroup $N \triangleleft \SL_2(\ZZ / p \ZZ)$, and the only possibilities are
\[
N \in \Big\{ \{I\}, \{\pm I\}, \PSL_2(\ZZ / p \ZZ), \SL_2(\ZZ / p \ZZ) \Big\}.
\]
Recall that $F=K_1 \cap K_2$ and $F \supset \QQ(\zeta_p)$. By Goursat's lemma and the Galois correspondence, we have that the index $[\SL_2(\ZZ / p \ZZ):N]$ is equal to the degree of the extension $[F:\QQ(\zeta_p)]$. Thus we see that $F$ is strictly larger than $\QQ(\zeta_p)$ if and only if $N \neq \SL_2(\ZZ / p \ZZ)$. 

If $N=\{I\}$, then in fact $F = K_1 K_2$; this is impossible, as $\ell_1$ ramifies in $K_1$ but not in $K_2$ by Proposition \ref{prop-Neron-Ogg-Sha} and the fact that $\ell_1\nmid \Delta(E_2)=\ell_2(432+\ell_2)$.

If $N=\langle \pm I \rangle$, then $[K_1:F]=2$. But by Lemma \ref{lem-ramification}, the ramification index of $\ell_1$ in $K_1 / \QQ(\zeta_p)$ is bigger than $2$, and $\ell_1$ is unramified in $K_2$ (and hence in $F$), so this impossible.

Finally, if $N=\PSL_2(\ZZ / p \ZZ)$, then $[F:\QQ(\zeta_p)]=2$, hence $F/\QQ$ is a degree $2(p-1)$ Galois extension. However, no such Galois subextensions of $K_1$ or $K_2$ exist; see the discussion and diagram following \cite[Remark 5.2.1]{Adelmann}.

Thus, the only possibility which our hypotheses allow is $N=\SL_2(\ZZ / p \ZZ)$ as desired. 

\end{proof}

\section{$p^2$-Division fields for $p=2,3$}\label{Section-p23}

In this section, we deal with Condition (2) of Lemma \ref{lem-2_conditions}, so given a pair $(E_1,E_2)$ as before, we now wish to show that
\begin{equation}\label{condition2}
\Im\bar\rho_{(E_1,E_2),36} = D_{36}.
\end{equation}
Similar to the setup in Section \ref{section-p5andup}, for $i=1,2$, let $K_{i,n}=\QQ(E_i[n])$ denote the $n$-Division field of $E_i$, which is the number field obtained by adjoining to $\QQ$ the $x$- and $y$- coordinates of the $n$-torsion points of $E_i$. Since $E_i$ is a Serre curve, we have
\[
\mathrm{Gal}(K_{i,36}/\QQ) \simeq \GL_2(\ZZ / 36 \ZZ).
\]
Once again, the Weil pairing forces an inclusion $\QQ(\zeta_{36}) \subset K_{i,36}$, where $\zeta_{36}$ is a primitive $36$-th root of unity. It follows that $K_{1,36} \cap K_{2,36} \supset \QQ(\zeta_{36})$ is a nontrivial extension of $\QQ$.  Just as in the $p \geq 5$ case, this implies that the Galois group $\mathrm{Gal}(L/\QQ)$ of the compositum $L=K_{1,36} K_{2,36}$ is a {\em proper} subgroup of $\big(\GL_2(\ZZ / 36 \ZZ) \big)^2$, determined (via Goursat's lemma) by a normal subgroup of $\GL_2(\ZZ / 36 \ZZ)$. Condition (\ref{condition2}) is equivalent to the statement that $K_{1,36}$ and $K_{2,36}$ are maximally disjoint in the sense that
\[
\Im\bar\rho_{(E_1,E_2),36} = D_{36} \Longleftrightarrow K_1 \cap K_2 = \QQ(\zeta_{36}).
\]

For $i=1,2$ Figure \ref{diagram-36-division-field} illustrates the decomposition of $K_{i,36}$ in terms of smaller division fields. The edges are marked by Galois groups, which are determined by the fact that $E_i$ is a Serre curve.

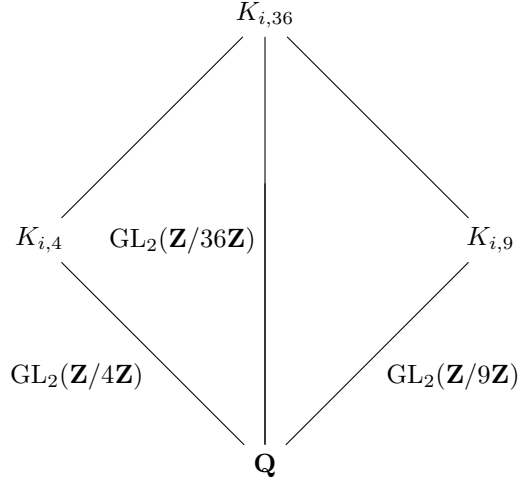
\begin{figure}
\centering
 \begin{tikzpicture}[node distance = 3cm, auto]
      \node (Q)       {$\QQ$};
      \node (K4) [above of=Q, left of=Q] {$K_{i,4}$};
      \node (K9) [above of=Q, right of=Q] {$K_{i,9}$};
      \node (K36)  [above of=Q, node distance = 6cm] {$K_{i,36}$};
      \draw[-] (Q) to node {} (F);
      \draw[-] (Q) to node {$\GL_2(\ZZ / 4 \ZZ)$} (K4);
      \draw[-] (Q) to node [swap] {$\GL_2(\ZZ / 9 \ZZ)$} (K9);
      \draw[-] (Q) to node {$\GL_2(\ZZ / 36 \ZZ)$} (K36);
      \draw[-] (K4) to node {} (K36);
      \draw[-] (K9) to node {} (K36);
      \end{tikzpicture}
      \caption{Decomposition of the 36-Division Fields for $E_i$}\label{diagram-36-division-field}
\end{figure}

Noting that $\GL_2(\ZZ / 36 \ZZ) \simeq \GL_2(\ZZ / 4 \ZZ) \times \GL_2(\ZZ / 9 \ZZ)$, we see that Figure \ref{diagram-36-division-field} and Goursat's lemma imply that $K_{i,4} \cap K_{i,9} = \QQ$. Furthermore, since $\mathrm{Gal}(L/\QQ)$ is a subgroup of $\mathrm{Gal}(K_{1,36} / \QQ) \times \mathrm{Gal}(K_{2,36} / \QQ)$, the same diagram shows that verifying (\ref{condition2}) is equivalent to verifying the following three assertions:
\begin{itemize}
\item $K_{1,4} \cap K_{2,4} = \QQ(\zeta_4)$;
\item $K_{1,9} \cap K_{2,9} = \QQ(\zeta_9)$;
\item $K_{i,4} \cap K_{j,9} = \QQ$ \ for $i \neq j$.
\end{itemize}

We now handle each case in turn. For the rest of the section, let $\Delta_i= - \ell_i (432\ell_i + 1)$ denote the discriminant of $E_i$. Just as in Section \ref{section-p5andup}, our arguments will depend crucially on our hypothesis that $\gcd(\Delta_1,\Delta_2)=1$.

\begin{lem}\label{lemma-4-division}
For our pair $(E_1,E_2)$, we have $K_{1,4} \cap K_{2,4} = \QQ(\zeta_4)$.
\end{lem}
\begin{proof}
The subfield structure of $4$-division fields of elliptic curves is explained in detail in \cite[Chapter 5.5]{Adelmann}. In particular, every subfield of $K_{i,4}$ which properly contains $\QQ(\zeta_4)$ also contains $\QQ(\zeta_4,\sqrt{\Delta_i})$, as well as {\em all} subfields which are quadratic over $\QQ$ . (See Figure \ref{diagram-4-division-field}.) Let $F=K_{1,4} \cap K_{2,4}$, so we have a containment $\QQ(\zeta_4) \subset F$. By the subfield diagram, if $[F:\QQ(\zeta_4)]>1$ then we must also have $\QQ(\zeta_4,\sqrt{\Delta_1}) \subset F \subset K_{2,4}$. But then this implies that the quadratic field $\QQ(\sqrt{\Delta_1})$ is also contained in $F \subset K_{2,4}$. However, the only quadratic subfields of $K_{2,4}$ are 
\[
\QQ(\zeta_4), \QQ(\sqrt{\Delta_2}), \mathrm{\text{ and  }} \QQ(\sqrt{-\Delta_2}).
\]
Since $\ell_1$ is an odd prime and $\gcd(\Delta_1,\Delta_2)=1$, we cannot have equality between $\QQ(\sqrt{\Delta_1})$ and any of the aforementioned fields. So we must have $[F:\QQ(\zeta_4)]=1$ which proves the lemma.  

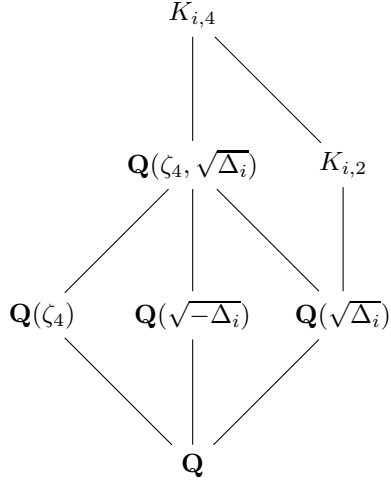
\begin{figure}
\centering
 \begin{tikzpicture}[node distance = 2cm, auto]
      \node (Q)       {$\QQ$};
      \node (Qzeta) [above of=Q, left of=Q] {$\QQ(\zeta_4)$};
      \node (Q-D) [above of=Q] {$\QQ(\sqrt{-\Delta_i})$};
      \node (QD) [above of=Q, right of=Q] {$\QQ(\sqrt{\Delta_i})$};
      \node (F)  [above of=Q-D] {$\QQ(\zeta_4, \sqrt{\Delta_i})$};
      \node (KE2) [above of=QD] {$K_{i,2}$};
      \node (K)  [above of=F] {$K_{i,4}$};
      \draw[-] (Q) to node {} (Qzeta);
      \draw[-] (Q) to node {} (QD);
      \draw[-] (Q) to node {} (Q-D);
      \draw[-] (Qzeta) to node {} (F);
	  \draw[-] (QD) to node {} (F);
	  \draw[-] (Q-D) to node {} (F);
	  \draw[-] (F) to node {} (K);
	  \draw[-] (QD) to node {} (KE2);
	  \draw[-] (KE2) to node {} (K);   
      \end{tikzpicture}
      \caption{A portion of the subfield diagram of $K_{i,4}$ from \cite[Figure 5.7]{Adelmann}}\label{diagram-4-division-field}
\end{figure}
\end{proof}

The argument for the case of $9$-division fields is very similar to that of Lemma \ref{lemma-4-division}. First we recall a result about the structure of the $3$-Division fields of elliptic curves.

\begin{lem}\label{lem-3-division}
Let $M=\QQ(x(E_i[3]))$ denote the number field obtained by adjoining to $\QQ$ the $x$-coordinates of the $3$-torsion points of $E_i$. Then $\QQ(\sqrt[3]{\Delta_i},\zeta_3)$ is the unique subfield of $M$ which has degree $6$ over $\QQ$. The only other subfield of $K_{i,9}$ which has degree $6$ over $\QQ$ is $\QQ(\zeta_9)$.
\end{lem}
\begin{proof}
The first statement is \cite[Proposition 5.4.3]{Adelmann}. The second statement is visible in \cite[Figure 5.4]{Adelmann}.
\end{proof}

\begin{lem}\label{lemma-9-division}
For our pair $(E_1,E_2)$, we have $K_{1,9} \cap K_{2,9} = \QQ(\zeta_9)$.
\end{lem}
\begin{proof}
The subfield structure of $9$-Division fields of elliptic curves is also explained in detail in \cite[Chapter 5.2]{Adelmann}. In particular, by \cite[Figure 5.4]{Adelmann} every subfield of $K_{i,9}$ which properly contains $\QQ(\zeta_9)$ also contains $\QQ(\zeta_3,x(E_i[3]))$

Let $F=K_{1,9} \cap K_{2,9}$, so we have a containment $\QQ(\zeta_9) \subset F$. If $[F:\QQ(\zeta_9)]>1$ then we must also have $\QQ(\zeta_3,x(E_1[3])) \subset F \subset K_{2,9}$. But then Lemma \ref{lem-3-division} implies that
\[
\QQ(\sqrt[3]{\Delta_1},\zeta_3) = \QQ(\sqrt[3]{\Delta_2},\zeta_3) 
\]
which is impossible since $\gcd(\Delta_1,\Delta_2)=1$. Thus $[F:\QQ(\zeta_9)]=1$ which proves the lemma.

\end{proof}

It remains to consider the possible entanglement between the $4$- and $9$-Division fields of our elliptic curves. By symmetry it suffices to show the following.

\begin{lem}\label{lemma-mixed-division}
For our pair $(E_1,E_2)$, we have $K_{1,4} \cap K_{2,9} = \QQ$.
\end{lem}

\begin{proof}

By \cite[Figure 5.4]{Adelmann}, every subextension of $K_{2,9}$ which is Galois over $\QQ$ contains $\QQ(\zeta_3)$ as the unique subextension which is quadratic over $\QQ$. Therefore, if $F= K_{1,4} \cap K_{2,9}$ satisfies $[F:\QQ]>1$, then $\QQ(\zeta_3) \subset F$. But also $F \subset K_{1,4}$, and as shown in Figure \ref{diagram-4-division-field}, the only quadratic subextensions of $K_{1,4}$ are
\[
\QQ(\zeta_4) , \QQ(\sqrt{-\Delta_1}), \mathrm{and}\ \QQ(\sqrt{\Delta_1}).
\]
One checks that if $\ell_1 = 3$ then $\Delta_1 = 3 \cdot 1297$; otherwise, $\ell_1 > 3$ and $\nu_{\ell_1}(\Delta_1)=1$, so in any case none of these extensions is equal to $\QQ(\zeta_3)$. It follows that $[F:\QQ]=1$, proving the lemma. 

\end{proof}

We summarize the results of this section.

\begin{prop}\label{36-prop}
For our chosen pair of elliptic curves $(E_1,E_2)$, we have
\[
\Im \bar\rho_{(E_1,E_2),36} = D_{36}
\]
\end{prop}
\begin{proof}
This follows immediately Lemma \ref{lemma-4-division}, Lemma \ref{lemma-9-division}, and Lemma \ref{lemma-mixed-division}.
\end{proof}

We can now prove the main result of this paper.

\begin{thm}\label{thm-main-body}
Let $\ell_1$ and $\ell_2$ be odd primes not equal to $7$ such that $\gcd(432\ell_1^2+\ell_1, 432\ell_2^2+\ell_2) = 1$, and for $i=1,2$ let 
\[ 
E_{\ell_i}: y^2+xy = x^3+\ell_i.
\]
Then the pair $(E_{\ell_1},E_{\ell_2})$ is a Serre pair.
\end{thm}
\begin{proof}
This follows immediately from Lemma \ref{lem-2_conditions}, Proposition \ref{prop-p-atleast-5}, and Proposition \ref{36-prop}.
\end{proof}

\section{Serre $k$-tuples}

Given a $k$-tuple of elliptic curves $(E_1,\ldots,E_k)$, one can generalize the above construction in the obvious way to obtain a representation
\[
\rho_{(E_1,\ldots,E_k)} : G_\QQ \to \big(\GL_2(\hat{\ZZ}) \big)^k,
\]
whose image is contained in 
\[
D^{(k)} := \left\{(A_1,A_2,\dots,A_k)\in \big(\GL_2(\hat{\ZZ}) \big)^k: \det A_1 = \det A_2 = \dots = \det A_k \right\}.
\]
Unsurprisingly, one has
\[
\left[D^{(k)} : \Im\rho_{(E_1,\ldots,E_k)}\right] \geq 2^k.
\]

\begin{dfn}
For any integer $k \geq 1$, a $k$-tuple $(E_1,\ldots, E_k)$ of elliptic curves is called a \emph{Serre $k$-tuple} if $[D^{(k)}: \Im\rho_{(E_1,\ldots,E_k)}] = 2^k$.
\end{dfn}

In \cite[Theorem 4.3]{Jones2013}, it is shown that almost all $k$-tuples of elliptic curves are Serre $k$-tuples. Theorem $\ref{thm-main-body}$ easily generalizes to the case $k \geq 2$.

\begin{thm}\label{thm-k-tuple}
Let $\ell_1,\ldots,\ell_k$ be odd primes not equal to $7$ such that $\gcd(432\ell_i^2 + \ell_i,432\ell_j^2 + \ell_j)=1$ for each pair $1 \leq i < j \leq k$. For each $1 \leq i \leq k$ let
\[ 
E_{\ell_i}: y^2+xy = x^3+\ell_i.
\]
Then $(E_{\ell_1},\ldots,E_{\ell_k})$ is a Serre $k$-tuple.
\end{thm}
\begin{proof}
Just as in the $k=2$ case, showing that $(E_{\ell_1},\ldots,E_{\ell_k})$ is a Serre $k$-tuple is equivalent to showing that the $E_{\ell_i}$ have maximally disjoint division fields \cite[Corollary 6.7]{JonesPreprint}. Since the discriminants of each elliptic curve in the $k$-tuple are pairwise relatively prime, Theorem \ref{thm-main-body} shows that the division fields for $E_{\ell_1},\ldots,E_{\ell_k}$ are pairwise maximally disjoint, and the result follows.
\end{proof}

\begin{rmk}
The argument in Corollary \ref{corollary-main}, applied inductively, shows that Theorem \ref{thm-k-tuple} produces infinitely many examples of Serre $k$-tuples.
\end{rmk}

\section{Final remarks}

Throughout this paper, we have relied on the elliptic curves
\[
E_i:= y^2 + xy = x^3 + \ell_i
\]
to prove Theorem \ref{thm-main-body}. However, a careful reading of our arguments reveals that only the following facts about the $E_i$ were used:
\begin{itemize}
\item $E_i$ is a Serre curve, and
\item $\Delta_i = \ell_i(432 \ell_i + 1)$
\end{itemize}
It is clearly necessary for the $E_i$ to be Serre curves, while precise knowledge of the discriminant of $E_i$ allowed us to compare the ramification of $\ell_i$ in various division fields. While Theorem \ref{thm-main-body} provides infinitely many explicit examples of Serre $k$-tuples, the arguments in this paper actually prove the following more general statement.

\begin{thm}\label{thm-general}
Let $E_1,\ldots,E_k$ be elliptic curves with discriminants $\Delta_1, \ldots,\Delta_k$, respectively. Suppose that each $E_i$ is a Serre curve, and that for $i=1,\ldots,k$ there exist odd primes $\ell_i > 3$ such that
\begin{itemize}
\item $v_{\ell_i}(\Delta_i) \equiv 1 \mod 2$; 
\item $E_i$ has split multiplicative reduction at $\ell_i$; and
\item $v_{\ell_i}(\Delta_j)=0$ for $i \neq j$.
\end{itemize}
Then $(E_1,\ldots, E_k)$ is a Serre $k$-tuple.
\end{thm}

\section*{Acknowledgments} The authors would like to thank \'Alvaro Lozano-Robledo and Sam Taylor for their helpful discussions throughout the writing process. 

\end{document}